\documentclass{amsart}
\usepackage {amsmath, amssymb, amscd, mathrsfs, url,pinlabel,hyperref, verbatim}
\usepackage[text={6.3in,8.2in},centering,letterpaper,dvips]{geometry}
\usepackage{color,multirow,dcpic,latexsym,pictexwd,graphicx,epstopdf}
\usepackage{comment}

\newtheorem {theorem}{Theorem}

\newtheorem {corollary}[theorem]{Corollary}

\theoremstyle{remark}

\newtheorem {remark}[theorem]{Remark}

\numberwithin{equation}{section}

\newcommand{\KCH}{\widetilde{HC}_0}
\newcommand{\R}{\mathfrak{R}}
\newcommand{\Z}{\mathbb{Z}}
\newcommand{\Tpq}{T_{p,q}}
\newcommand{\pihat}{\hat{\pi}}

\author{Cameron Gordon and Tye Lidman}
\date{}
\title{Knot contact homology detects cabled, composite, and torus knots}

\begin{document}

\begin{abstract}
Knot contact homology is an invariant of knots derived from Legendrian contact homology which has numerous connections to the knot group.  We use basic properties of knot groups to prove that knot contact homology detects every torus knot.  Further, if the knot contact homology of a knot is isomorphic to that of a cable (respectively composite) knot, then the knot is a cable (respectively composite).
\end{abstract}

\maketitle

Associated to a knot $K \subset \mathbb{R}^3$, knot contact homology is a combinatorial invariant which arises from constructions in contact and symplectic geometry \cite{KCH1, KCH2, FKCH, EENS}.  More specifically, the conormal bundle $\Lambda_K$ of $K$ in the unit cotangent bundle $ST^*\mathbb{R}^3$ is a Legendrian submanifold, and knot contact homology comes from the Legendrian contact homology of $\Lambda_K$.  This invariant is able to distinguish mirrors and mutant pairs, determines the Alexander polynomial, and detects the unknot \cite{FKCH}.  
In fact, it is still open as to whether knot contact homology is a complete knot invariant.\footnote{Degree zero knot contact homology is isomorphic to degree zero string homology, an object defined by Cieliebak, Ekholm, Latschev, Ng in \cite{CELN}.  It is shown there that a variant of degree zero string homology is a complete invariant.}  

In this paper, we will show that knot contact homology detects each of the torus knots as well as being a cable or composite.  Here, we will work with a version of the fully noncommutative degree zero knot contact homology with $U=1$, which we denote by $\KCH(K)$; for comparison with other appearances in the literature, this is denoted by  $H^{\operatorname{contact}}_0(K)$ in \cite{CELN} and $\widetilde{HC}_0(K)|_{U=1}$ in \cite{TopKCH}. 

\begin{theorem}\label{thm:main}
Let $K$ be an oriented knot in $\mathbb{R}^3$ and let $\Tpq$ denote the $(p,q)$-torus knot.  If $\KCH(K) \cong \KCH(\Tpq)$, then $K$ is isotopic to $\Tpq$.  
\end{theorem}

Since knot contact homology is an invariant of Legendrian isotopy, we obtain the following.
\begin{corollary}
If $\Lambda_K$ is Legendrian isotopic to $\Lambda_{\Tpq}$ in $ST^*\mathbb{R}^3$, then $K$ is (smoothly) isotopic to $\Tpq$.  
\end{corollary}

Using similar techniques, we will also show 
\begin{theorem}\label{thm:cable-composite}
Let $K$ be an oriented knot in $\mathbb{R}^3$ and let $C_{p,q}(J)$ denote the $(p,q)$-cable of a knot $J$.  If $\KCH(K) \cong \KCH(C_{p,q}(J))$, then $K$ is isotopic to a $(p',q')$-cable of a knot $J'$ where $pq = p'q'$.  On the other hand, if $\KCH(K) \cong \KCH(J)$ and $J$ is composite, then so is $K$.
\end{theorem}

The key starting point for Theorem~\ref{thm:main} is a forthcoming theorem of Cieliebak, Ekholm, Latschev, and Ng which relates $\KCH$ to the knot group \cite{CELN}.  (A sketch is given in \cite[Theorem 4.9]{TopKCH}.)  We first set our notation before stating their theorem.  Since degree zero knot contact homology detects the unknot \cite[Proposition 5.10]{FKCH}, we will assume throughout that all knots are oriented and non-trivial to simplify the discussion.  (See Remark~\ref{rmk:unknot} for a discussion of the unknot.)  Note that the orientation induces canonical representatives for the Seifert framing and the meridian, $\lambda_K$ and $\mu_K$ respectively, in the knot group $\pi_K = \pi_1(S^3 \setminus K)$.  We let $\pihat_K$ denote the peripheral subgroup of $\pi_K$, i.e., the subgroup generated by $\mu_K$ and $\lambda_K$.  Of course $\mathbb{Z}[\pi_K]$ contains $\Z[\pihat_K] \cong \Z[\mu^{\pm 1}_K, \lambda_K^{\pm 1}]$ as a subring.  Finally, let $\mathfrak{R}_K$ be the subring of $\mathbb{Z}[\pi_K]$ generated by $\mathbb{Z}[\pihat_K]$ and $\{\gamma - \mu_K \gamma \mid \gamma \in \pi_K\}$.  

Degree zero knot contact homology, $\KCH(K)$,  takes the form of a ring equipped with an embedding of $\Z[\mu^{\pm 1}, \lambda^{\pm 1}]$ into $\KCH(K)$.  An isomorphism between $\KCH(K)$ and $\KCH(K')$ is a ring isomorphism which restricts to the identify on the subrings $\Z[\mu^{\pm 1}, \lambda^{\pm 1}]$.  We are now ready to state a powerful relationship between $\KCH(K)$ and $\pi_K$.                   

\begin{theorem}[Cieliebak-Ekholm-Latschev-Ng, \cite{CELN}]\label{thm:kch-iso}
Let $K$ be a non-trivial, oriented knot.  Then, there is an isomorphism between $\KCH(K)$ and $\R_K$ which sends $\mu$, $\lambda$ to $\mu_K$, $\lambda_K$ respectively.  
\end{theorem}

Theorems~\ref{thm:main} and \ref{thm:cable-composite} will follow easily from Theorem~\ref{thm:kch-iso}, when we combine this with the fact that knot groups are locally indicable, which follows from \cite[Lemma 2]{HowieShort}, and the classical result of Higman that group rings of locally indicable groups have no zero-divisors \cite{Higman}.

\begin{remark}\label{rmk:unknot}
The degree zero knot contact homology of the unknot is given by $\mathbb{Z}[\mu^{\pm 1},\lambda^{\pm 1}]/(1-\mu)(1-\lambda)$.  It is pointed out in \cite{CELN} that since $\KCH(U)$ has zero-divisors and $\KCH(K)$ has no zero-divisors for non-trivial $K$ by Theorem~\ref{thm:kch-iso}, we obtain an alternate proof that degree zero knot contact homology detects the unknot.   
\end{remark}

\begin{proof}[Proof of Theorem~\ref{thm:main}]
As discussed, degree zero knot contact homology detects the unknot, so we suppose throughout that $K$ is non-trivial.    Let $T = \Tpq$ and suppose that $\KCH(K) \cong \KCH(T)$.  First, we show that $K$ must be equivalent to $T_{p',q'}$ where $p'q' = pq$.    By Theorem~\ref{thm:kch-iso}, we have a ring isomorphism $\psi$ between $\R_T$ and $\R_K$ which sends $\mu_{T}, \lambda_{T}$ to $\mu_K, \lambda_K$ respectively.  We will focus in particular on the elements $\phi_K = \mu^{pq}_K\lambda_K$ and $\phi_T = \mu^{pq}_T \lambda_T$, which are identified via $\psi$.

Recall that $\pi_T$ has non-trivial center, isomorphic to $\langle \phi_T \rangle$, since this is the fiber slope coming from the Seifert structure on the exterior of $T$.  On the other hand, by work of Burde-Zieschang \cite{Center}, if $K$ is not a torus knot then $\pi_K$ has trivial center.  Therefore, if $K$ is not isotopic to $T_{p',q'}$ with $pq = p'q'$, we have that $\phi_K$ is not central in $\pi_K$.  Thus, there exists $\gamma \in \pi_K$ such that $\gamma \phi_K \neq \phi_K \gamma$.  However, since $\KCH(K) \cong \KCH(\Tpq)$, we have that $\phi_K$ is central in $\R_K$.  Thus, $\phi_K (\gamma - \mu_K \gamma) = (\gamma - \mu_K \gamma) \phi_K$.  In $\Z[\pi_K]$, we rearrange to obtain
\begin{align*}
\phi_K \gamma - \gamma \phi_K &= \phi_K \mu_K \gamma - \mu_K \gamma \phi_K \\
&= \mu_K (\phi_K \gamma - \gamma \phi_K),   
\end{align*}     
where the second equality comes from the fact that $\mu_K$ and $\phi_K$ commute, being elements of $\pihat_K \cong \mathbb{Z}^2$.  Again, we rearrange to obtain, in $\Z[\pi_K]$, the equality $(\mu_K-1) (\phi_K \gamma - \gamma \phi_K) = 0$.  
Note that $\mu_K \neq 1$, since $\mathbb{Z}[\pi_K]$, as an abelian group, is freely generated by the elements of $\pi_K$.  Consequently, $\mathbb{Z}[\pi_K]$ has zero-divisors, since $\phi_K \gamma \neq \gamma \phi_K$.  As discussed above, $\pi_K$ is locally indicable, and therefore $\mathbb{Z}[\pi_K]$ has no zero-divisors.  This is a contradiction.  It therefore follows that $K = T_{p',q'}$, where $p'q' = pq$.  In particular, we point out that $K$ is not isotopic to $T_{-p,q}$.   

It remains to show degree zero knot contact homology distinguishes $T_{p',q'}$ from $\Tpq$ where $p'q' = pq$, but $\pm \{p,q\} \neq \{p',q'\}$.  If $\Tpq$ and $T_{p',q'}$ had isomorphic degree zero knot contact homology, then they would have the same stable $A$-polynomial $\widetilde{A}_K(\mu,\lambda)$ by \cite[Corollary 1.5]{Cornwell2} and \cite{FKCH}.  However, Cornwell computes $\widetilde{A}_{\Tpq}$ in \cite[Theorem 1.3]{Cornwell}; after normalizing this polynomial so that it is not divsible by $\mu$ or $\lambda$, the smallest degree non-constant term of $\widetilde{A}_{\Tpq}$ has exponent $|p||q| - |q| + 1$.  Since $pq = p'q'$, by assumption, the result now follows.  
\end{proof}

In order to detect cables and composite knots, we recall the result of Simon \cite{Simon}, generalizing the characterization of torus knots in terms of centralizers; this states that if there exists a non-trivial element $v \in \pihat_K$ and $g \in \pi_K \setminus \pihat_K$ with $vg = gv$, then $K$ is either a cable or a composite knot.  In fact, the proof in \cite{Simon} yields: 

\begin{theorem}\label{thm:peripheral-commutes}
Suppose a non-trivial element $v \in \pihat_K$ commutes with $g \in \pi_K \setminus \pihat_K$.  Then, either:
\begin{enumerate}
\item $K$ is composite and $v$ is a power of $\mu$, or 
\item $K$ is a $(p,q)$-cable and $v$ is a power of $\mu^{pq} \lambda$.  
\end{enumerate}
\end{theorem}

With this we are now able to prove Theorem~\ref{thm:cable-composite}.  
\begin{proof}[Proof of Theorem~\ref{thm:cable-composite}]
For notational simplicity, we give only the argument for cabled knots.  In light of Theorem~\ref{thm:peripheral-commutes}, it will be clear that the same argument applies for composite knots as well.  

As before, we may assume that $K$ is a non-trivial knot.  Suppose that $K$ is not a cabled knot but $\KCH(K) \cong \KCH(C_{p,q}(J))$ for some knot $J$.  Let $g \in \pi_{C_{p,q}(J)}$ be a non-peripheral element which commutes with $\mu^{pq} \lambda$.  In $\KCH(C_{p,q}(J))$, we see that $\mu^{pq}\lambda$ commutes with $(1 - \mu) g$.  Let $\psi: \KCH(C_{p,q}(J)) \to \KCH(K)$ denote the isomorphism.  Note that we can write $\psi((1-\mu)g)$ as $z + (1-\mu)w$, where $z \in \Z[\pihat_K]$ and $w = \sum^n_{i = 1} a_i w_i$, where $a_i$ are non-zero integers and $w_i$ are distinct elements in $\pi_K \setminus \pihat_K$.  Of course $z$ and $\mu^{pq}\lambda$ commute, so we see that $\mu^{pq}\lambda$ and $(1-\mu)w$ commute in $\KCH(K)$ and thus in $\Z[\pi_K]$.  As in the proof of Theorem~\ref{thm:main}, since $(1-\mu) \neq 0$, we have 
\[
\sum^n_{i=1} a_i (\mu^{pq}\lambda w_i - w_i \mu^{pq}\lambda) = 0.  
\]
Note in particular that for each $i$, $\mu^{pq} \lambda w_i = w_j \mu^{pq} \lambda$ for some $j$.  It follows that there exists $k > 0$ such that $(\mu^{pq} \lambda)^k w_1 = w_1 (\mu^{pq} \lambda)^k$.  Since $w_1$ is not peripheral, it now follows from Theorem~\ref{thm:peripheral-commutes} that $K$ is a $(p',q')$-cable with $p'q' = pq$.       
\end{proof}

It is an interesting problem to try to determine if $\KCH(K) \cong \KCH(C_{p,q}(J))$ implies that $K$ must in fact be a $(p,q)$-cable of a knot $J'$ with $\KCH(J') \cong \KCH(J)$.  A similar question exists for composite knots.  
  

\begin{remark}
In fact, the proofs of Theorems~\ref{thm:main} and \ref{thm:cable-composite} use a weaker property than local indicability of the knot group.  The proofs only use that left-multiplication by $(1-\mu)$ is an injection on the group ring of any knot group.  That this property holds can be seen by a similar argument as at the end of the proof of Theorem~\ref{thm:cable-composite}.
\end{remark}

\section*{Acknowledgments} We would like to thank Robert Lipshitz for introducing us to this problem and for his encouragement.  We would also like to thank Chris Cornwell and Lenny Ng for help with knot contact homology and comments on an earlier version of this note.  

\bibliographystyle{alpha}
\bibliography{biblio}

\end{document}